\documentclass[12pt]{amsart}
\usepackage{amsmath,amssymb,amsfonts, amscd}
\usepackage{amsbsy}
\usepackage{amsthm}
\usepackage{amstext}
\usepackage{amsopn}
\usepackage{hyperref}
\usepackage{mathrsfs}
\usepackage{graphicx}
\usepackage{latexsym}
\usepackage[T1]{fontenc}
\usepackage{color}
\usepackage[abs]{overpic}

\newcommand{\Hmm}[1]{\leavevmode{\marginpar{\tiny%
$\hbox to 0mm{\hspace*{-0.5mm}$\leftarrow$\hss}%
\vcenter{\vrule depth 0.1mm height 0.1mm width \the\marginparwidth}%
\hbox to 0mm{\hss$\rightarrow$\hspace*{-0.5mm}}$\\\relax\raggedright #1}}}
\newcommand{\ran}{\mathrm{Ran}}

\newtheorem{thm}{Theorem}[section]

\newtheorem{lemma}[thm]{Lemma}
\newtheorem{pro}[thm]{Proposition}

\theoremstyle{definition}
\newtheorem*{defi}{Definition}

\newcommand{\R}{{\mathbb R}}

\newcommand{\N}{{\mathbb N}}

\newcommand{\de}{{\delta}}
\newcommand{\eps}{{\varepsilon}}
\newcommand{\gm}{{\gamma}}

\newcommand{\supp}{{\mathrm{supp}\,}}

\begin{document}
\title[Hopf-Rinow]{A  discrete Hopf-Rinow-theorem}

\author[M. Keller]{Matthias Keller}
\author[F. M\"unch]{Florentin  M\"unch}
\address{Matthias Keller, Florentin M\"unch, Universit\"at Potsdam, Institut f\"ur Mathematik, 14476  Potsdam, Germany}
\email{matthias.keller@uni-potsdam.de\\chmuench@uni-potsdam.de}


\date{\today}
\maketitle

\begin{abstract}
We prove a version of the Hopf-Rinow-theorem with respect to path metrics on discrete spaces. The novel aspect is that we do not a priori assume local finiteness but isolate a local finiteness type condition, called essential local finiteness,  that is indeed necessary. As a side product  we identify the maximal weight, called the geodesic weight, which generates the path metric in the situation when the space is complete with respect to any of the equivalent notions of completeness proven in the Hopf-Rinow theorem. As an application, we characterize the graphs for which  the resistance metric is a path metric induced by the graph structure.
\end{abstract}

\section{Introduction}
In 1931  Hopf and Rinow, \cite{HopfRinow}, have proven a most fundamental theorems in differential geometry about the equivalence of metric and geodesic completeness, as well as compactness of distance balls on Riemannian manifolds. For a modern reference we refer to \cite{Jost17}.

Recently, the investigation of metrics on graphs have gained strong momentum. Specifically, path metrics play an important role in the study of operator theory \cite{CdVTT11,HKMW13,Mil11}, spectral geometry \cite{BHK13,HKW13} and the heat equation \cite{BHY17,Fol11,Fol14,Hua14}. While the interest was first in locally finite graphs only, nowadays general weighted graphs gain more importance \cite{Ke15} as a crucial example class of non-local Dirichlet forms and their associated jump processes. Thus, geometric results which establish the analogy to the continuum setting of Riemannian manifolds are essential.

As for a Hopf-Rinow theorem, first discrete versions  have been proven in  \cite{Mil11} and \cite{HKMW13}. The argument given in \cite{Mil11} is based on length spaces in the sense of
Burago-Burago-Ivanov \cite{BBI01} and, while not mentioned explicitly, the length spaces in question are metric graphs
associated to discrete graphs. On the other hand, the proof given in \cite{HKMW13} is purely combinatorial.

In both of these,  the assumption of local finiteness of is central. Specifically in \cite{HKMW13}, examples are given which show that this assumption can not simply be removed.

The purpose of this article is to investigate the question to which extent this assumption is also necessary. The answer is that we only need a property to which we refer to as \emph{essential local finiteness}.  Essential local finiteness means that for every vertex the number of vertices which are ``close'' is  finite which will be made precise below. This way, the assumption of essential local finiteness becomes part of the characterization.
As a side product, we identify the maximal weight inducing the metric in the case when the path metric space is complete.

Furthermore, as an application we characterize when the resistance metric is a path metric induced by the graph structure.

\textbf{Acknowledgments.} Funding through the DFG is acknowledged by the authors.

\subsection{Set-up and definitions}

Throughout the  article, let $X$ be a countable set whose elements are called \emph{vertices}. We define the \emph{path space} of $X$ via
\[
\Pi(X) := \{\gamma: \{0,\ldots,n\} \to X \; | \;  n \in \N_0, \gamma \mbox{ injective}  \}
\]
and the \emph{space of infinite paths} of $X$ via
\[
\Pi_\infty (X):= \{\gamma: \N_0 \to X \; | \;  \gamma \mbox{ injective}  \}.
\]
We identify paths $ \gm $ in $ \Pi(X) $ (or $ \Pi_{\infty}(X) $) with $(\gamma(0),\ldots,\gamma(n))$ (or   with $(\gamma(0),\gamma(1),\ldots)$). For $\gamma = (x_0,\ldots,x_n) \in \Pi(X)$, we say $\gamma$ is a \emph{path} from $x_0$ to $x_n$.

	 We call a function $w:X\times X \to [0,\infty]$ a \emph{weight function}  on $X$ if $w$ is symmetric, i.e., $w(x,y)=w(y,x)$ for all $x,y \in X$, and zero on and only on the diagonal, i.e., $w(x,y)=0$ if and only if $ x=y $ for $x,y \in X$. 

Every weight $w$ on $X$ induces a pseudo metric $\delta_w:X\times X \to [0,\infty]$ on $X$ via
\[
\delta_w(x,y) := \inf    \left\lbrace   \sum_{i=1}^n w(x_{i-1},x_i) \mid (x_0,\ldots,x_n) \mbox{ path from } x \mbox{ to } y  \right\rbrace .    
\]
A \emph{pseudo metric} is a symmetric map that is zero on the diagonal and satisfies the triangle inequality. In contrast to a metric we allow for the value $ \infty $ and do not ask for definiteness.

Observe that $\delta_w \leq w$ (i.e., $\delta_w(x,y) \leq w(x,y)$ for all $x,y\in X$) for all weights $w$. Let $ \delta $ be a pseudo metric on $ X $. A weight $w$ on $X$ is said to \emph{generate} $\delta$ if $\delta = \delta_w$.
It is not hard to see from the triangle inequality that every pseudo metric $\delta$ is also a weight and $\delta_\delta=\delta$. A pseudo metric $ \delta $ is said to be \emph{discrete} if $ (X,\delta) $ is a discrete  space.

A weight function is called \emph{locally finite} if for all $x \in X$ we have $\#\{y\in X \mid w(x,y) < \infty\} < \infty$ . The following weaker notion of local finiteness is crucial for the considerations of this paper.

\begin{defi}[Essential local finiteness]
A weight function is called \emph{essentially locally finite} if 
for all $x \in X$ and all $R>0$
$$ \#\{y\in X \mid w(x,y) <R\} < \infty .$$ 
\end{defi}
For a weight $w$ and $\gamma=(x_0,\ldots,x_n) \in \Pi(X)$, we denote the $w$-\emph{length} by $$ l_w(\gamma) :=\sum_{i=1}^n w(x_{i-1},x_i). $$
Similarly, for $\gamma=(x_0,x_{1},\ldots ) \in \Pi_\infty(X)$,  we set  $l_w(\gamma) :=\sum_{i=1}^\infty w(x_{i-1},x_i)$.

A  $w$-\emph{geodesic} is a path $\gamma\in \Pi(X)$  from $x$ to $y$ such that 
$$ l_w(\gamma)= \delta_w(x,y).  $$
An infinite path $\gamma=(x_0,x_{1},\ldots) \in \Pi_\infty (X)$ is called a $w$-\emph{geodesic} from $x_0$ if  $(x_0,\ldots,x_n)$ is $ w $-geodesic for all $n \in\N$.
Observe that every $w$-geodesic is $\delta_w$-geodesic. 
A weight $ w $ is  \emph{geodesically complete} if $ l_{w}(\gamma)=\infty $  for all infinite $ w $-geodesics $\gamma$.

For a given pseudo metric space $(X,\delta)$, we write for short, $l:=l_\delta$ and call every $\delta$-geodesic a \emph{geodesic}.  A pseudo metric space $(X,\delta)$ is called \emph{geodesically complete} if $l(\gamma) = \infty$ for all infinite geodesics $\gamma$.
For a pseudo metric space $(X,\delta)$, we introduce the \emph{geodesic weight} $w_\delta$ via
\[
w_\delta(x,y) :=  \begin{cases}
								 \delta(x,y) &\mbox{if } (x,y) \mbox{ is the only geodesic from } x \mbox{ to } y  \\
								 \infty      & \mbox{else}.
				  \end{cases}
\]
Finally, given a pseudo metric $ \de $ on $ X $ we denote the distance balls for $ x\in X $ and $ R\ge 0 $ by
$ B_{R}(x):=\{y\in X\mid \de(x,y)\leq R\}. $

\section{The Hopf-Rinow theorem}

The Hopf-Rinow theorem below is the main result of the paper.

\begin{thm}[Hopf-Rinow theorem] \label{t:main}
	Let $(X,\delta)$ be a pseudo metric space. Then, there exists a weight generating $\delta$. Let $w$ be such a weight. Then the following statements are equivalent:
	\begin{enumerate}
		\item[(i)] $\#B_R(x) < \infty$ for all $x \in X$ and $R>0$.
		\item[(ii)] The space $(X,\delta)$ is complete and $w$ is essentially locally finite.
		\item[(iii)] The space $(X,\delta)$ is geodesically complete and the geodesic weight $w_\delta$ is essentially locally finite.
		\item [(iv)] The weight $ w $ is geodesically complete and  essentially locally finite.
	\end{enumerate}
	If one of these conditions holds, then $ \de $ is discrete and   for any $ x,y\in X $, there exists a $ w $-geodesic. Furthermore, 
	\[
	w_\delta(x,y) = \sup\{\widetilde w(x,y)\mid \mbox{$ \widetilde w $ weight such that } \delta_{\widetilde w} = \delta\}, \]
	$  \mbox{ for all } x,y \in X $ and $w_\delta$ generates $\delta$.
\end{thm}

We start the proof with a basic lemma which shows that essential local finiteness implies discreteness.

\begin{lemma} \label{Lemma discrete}
	Let $w$ be an essentially locally finite weight on $X$. Then, $\delta_w$ is discrete. 
\end{lemma}
\begin{proof}
Let $x \in X$. Since $\#\{y\in X\mid w(x,y) < 1\} < \infty$ and $ w $ vanishes only on the diagonal, there exists $\eps>0$ such that $w(x,y) > \eps$ for all $y \in X$.
Let $y \in X$, let $n \in \N$ and let $(x_0,\ldots,x_n)$ be a path from $x$ to $y$. Then, $\sum_{i=1}^n w(x_{i-1},x_i) \geq w(x,x_1) > \eps$.
Hence, $\delta_w(x,y) \geq \eps$ and thus, $B_{\eps/2}(x) = \{x\}$. Since $x$ was chosen arbitrary, this implies discreteness of $(X,\delta)$.
\end{proof}

The next lemma shows another consequence of essential local finiteness. Specifically, every infinite set of finite paths with bounded lengths gives rise to an infinite path of bounded length.

\begin{lemma}  \label{Lemma path}
	
	Let $w$ be an essentially locally finite weight on $X$ generating a metric $\delta=\delta_w$, let $x \in X$ and let $R>0$.  Let $\mathcal{P} \subseteq \{\gamma \in \Pi(X)\mid \gamma(0) = x, l_w(\gamma) < R\}$ be an infinite set of finite paths with fixed starting point and bounded $w$-length. Then, there exists an infinite path $\gamma_\infty \in \Pi_\infty(X)$ such that for every $n\in \N$ there exist infinitely many paths $\gamma \in \mathcal P$ with $\gamma(k)=\gamma_\infty(k)$, for $k=0,\ldots,n$, and $l_w(\gamma_\infty) \leq R$.
\end{lemma}

\begin{proof} Due to countabilty, we can assume wellordering of $X$ without applying the axiom of choice, and thus, every non-empty subset of $X$ possesses a minimum.

We define $\gamma_{\infty}$ inductively starting with $ \gm_{\infty}(0)=x $ and
$$\gamma_\infty(n+1) :=  \min\big\{y \in X\mid  \#\{\gamma \in \mathcal P_n \mid \gamma(n+1)=y \} =\infty  \big\}$$
for $ n\in\N $  with
$$
\mathcal P_n :=  \{\gamma \in \mathcal P\mid \gamma(i) = \gamma_\infty(i) \mbox{ for } i=0,\ldots,n \},
$$
whenever the set from which the minimum is taken is non-empty. Indeed, this is always the case as we show now: 

We know that the set of vertices which occur in paths of $ \mathcal{P} $ is finite  since $w(\gamma(i-1),\gamma(i)) < R$ for all $\gamma \in \mathcal P$ and since $w$ is essentially locally finite. But  every partition of an infinite set into finitely many subsets has an infinite subset.
This implies the existence of some $y$ such that there are infinitely many paths in $\mathcal{P}$ starting with $(\gamma_\infty(0),\ldots,\gamma_\infty(n),y)$ under the assumption of the existence of infinitely many paths in $\mathcal P$  starting with $(\gamma_\infty(0),\ldots,\gamma_\infty(n))$. By induction, we obtain that for every $n\in\N$, there exist infinitely many paths in $\mathcal P$ starting with $(\gamma_\infty(0),\ldots,\gamma_\infty(n))$. Since the $w$-length of all paths in $\mathcal P$ is upper bounded by $R$, we conclude $l_w(\gamma_\infty) \leq R$.
\end{proof}

\begin{lemma} \label{Lemma geodesic}
	Let $w$ be an essentially locally finite weight on $X$ such that $(X,\delta_w)$ is complete. Then, for all $x,y \in X$, there exists a $w$-geodesic from $x$ to $y$.
\end{lemma}

\begin{proof}
Let $x,y \in X$. Suppose, there is no $w$-geodesic from $x$ to $y$. Then $l_w(\gamma)>\delta_w(x,y)$ for all paths $\gamma$ from $x$ to $y$. But there is a sequence of paths $(\gamma_i)$ with $l_w(\gamma_i) {\longrightarrow} \delta_w(x,y)$ as $ {i \to \infty} $.
Hence, 
\[ 
\mathcal P :=\{\gamma \in \Pi(X) \mid   \mbox{ path from } x \mbox{ to }y \mbox{ such that } \,l_w(\gamma) < \delta_w(x,y)+1\}
\]
satisfies $\# \mathcal P = \infty$. 
Now,  Lemma~\ref{Lemma path}  yields the existence of an infinite path $\gamma_\infty$ with finite $w$-length $l_w(\gamma_\infty) \leq \delta_w(x,y) + 1$.  This implies that $\gamma_\infty$ is a Cauchy-sequence. Due to completeness, $\gamma_\infty$ converges to some limit in $ X $ which is not a discrete point since $ \gm_{\infty} $ is injective. This is a contradiction to Lemma~\ref{Lemma discrete} and thus, there exists a $w$-geodesic from $x$ to $y$.
\end{proof}

\begin{lemma}\label{lf&geodcompl=>complete}
	Let $ w $ be an essentially locally finite weight on $ X $ which is geodesically complete. Then, for every $ x,y\in X $ there is a $ w$-geodesic from $ x $ to $ y $.
\end{lemma}
\begin{proof}Let $ x,y\in X $ such that $ x\neq y $. Our aim is to find a $ w $-geodesic from $ x $ to $ y $. 
	To this end, let $ ( \gm_{j}) $ be a sequence of paths from $ x$ to $y $ such that
	\begin{align*}
	\lim_{j\to\infty}l_{w}( \gm_{j})=\de(x,y).
	\end{align*}

	Let $\mathcal P:=\{\gamma_j \mid j \in \N \}$.
	
	If $\#\mathcal P < \infty$, then $ \mathcal{P} $ obviously contains a $ w $-gedeosic from $ x $ to $ y $.
	
	If $\#\mathcal P = \infty$, then by Lemma~\ref{Lemma path} there is an infinite path $ \gm_{\infty} $ such that for every $ n\in\N $, there exist infinitely many $ \gm_{j}$ such that $ \gm_{j}(k)=\gamma_{\infty}(k) $ for all $ k=0,\ldots,n $. We fix $ n\in \N $ and let $ \eps>0 $. 
	Then, there exists $ \gm_{j} $  such that $ \gm_{j}(k)=\gamma_{\infty}(k) $ for all $ k=0,\ldots,n $ and $ l_{w}(\gm_{j}) \leq \de(x,y)+\eps $.

	 We estimate
	 \begin{align*}
	 \de(x,y)&\leq \de(\gm_\infty({0}),\gm_\infty({n}))+\de(\gm_\infty({n}),y)\\
	 &\leq  l_{w}((\gm_\infty({0}),\ldots,\gm_\infty({n})))+\de(\gm_\infty({n}),y)\\
	 &\leq l_{w}(\gm_{j})\\
	 &\leq \de(x,y)+\eps.
	 \end{align*}
	 Letting $ \eps \to 0$, we see that all inequalities turn into equalities and, therefore,
	 \begin{align*}
	 \de(\gm_\infty({0}),\gm_\infty({n}))= l_{w}((\gm_\infty({0}),\ldots,\gm_\infty({n}))).
	 \end{align*}

	 Hence, $ \gm_\infty $ is a $ w $-geodesic which has finite length. This contradicts the assumption of $ w $-geodesic completeness.
\end{proof}

\begin{proof}[Proof of Theorem~\ref{t:main}]
The existence of a weight $w$ that generates $\delta$ follows from $\delta_\delta=\delta$. Let $ w $ be such a weight.\smallskip

	(i) $\Rightarrow$ (ii). 	
The space  $(X,\delta)$ is complete since all balls are complete due to finiteness.
Since $\#B_R(x) < \infty$ for all $x \in X$ and $R>0$ and $w \geq \delta_w = \delta$, the weight $w$ is essentially locally finite.
 \smallskip

	(ii) $\Rightarrow$ (iii). 
We first show geodesic completeness.	
Suppose not. Then, there is an infinite geodesic $\gamma=(x_0,x_{1},\ldots)$ with $l(\gamma) < R$ for some $R>0$.
Then, $\gamma$ is a Cauchy sequence converging to some $x\in X$ due to completeness. But since $\gamma$ is injective, this implies that $x$ is not a discrete point. This is a contradiction to Lemma~\ref{Lemma discrete}. Hence, $X$ is geodesically complete.

Next, we show $w_\delta(x,y) \geq w(x,y)$ for all $x,y \in X$.
Let $x,y \in X$. Due to Lemma~\ref{Lemma geodesic}, there exists a $w$-geodesic $(x_0,\ldots,x_n)$ from $x=x_{0}$ to $y=x_{n}$. If $ n=1 $ then $w(x,y)=\delta(x,y)$  and if 
 $n\ge2$ then $w_\delta(x,y)=\infty$. Thus, $w_\delta(x,y)\in \{w(x,y),\infty\}$ which proves  $w_\delta(x,y) \geq w(x,y)$. We infer $$  \#\{y\in X\mid w_\delta(x,y) <R\} \leq \#\{y\in X\mid w(x,y) <R\} < \infty $$ for all $R>0$ by essential local finiteness of $ w $ which is assumed in (ii). This proves essential local finiteness of $w_\delta$.\smallskip

	(iii) $\Rightarrow$ (i). 
First, we show that $w_\delta$ generates $\delta$.  
To do so, we show that for all $x \neq y$, there exists a maximal $\delta$-geodesic $\gamma =(x_0,\ldots,x_n)$ from $x$ to $y$ in the sense that $(x_i,x_{i-1})$ is the only geodesic between $x_i$ and $x_{i-1}$ for $i=1,\ldots,n$.
Suppose not. 
Then, there is a sequence of geodesics $\gamma_k: \{0,\ldots,k\} \to X$ such that $\ran(\gamma_k) \subset \ran(\gamma_{k+1})$ for $k\geq 1$ where we denote the range of a function $f:A \to B$ by $\ran(f):=f(A)$.
The map $$ \Phi: \Gamma := \bigcup_{k\in\N} \ran (\gamma_k) \to [0, \delta(x,y)],\qquad z \mapsto \de(x,z)$$ is isometric since all $\gamma_k$ are geodesics. Since $\# \Gamma = \infty$, the range $\ran(\Phi)$ has an accumulation point due to the Bolzano-Weierstrass theorem and there exists a strictly monotone sequence $\left(r_k \right) $  in $ \ran(\Phi)$ converging to this accumulation point. Hence, $\left(\Phi^{-1}(r_k) \right)_{k=0}^\infty$ is an infinite geodesic of finite length. This is a contradiction to geodesic completeness and thus, there exists  a maximal $\delta$-geodesic $\gamma = (x_0,\ldots,x_n)$ from $x$ to $y$.
Moreover, maximality implies $w_\delta(x_i,x_{i-1}) = \delta(x_i,x_{i-1})$ for $i=1,\ldots,n$ and thus, $l_{w_\delta}(\gamma) = \delta(x,y)$. Consequently, $w_\delta$ generates $\delta$ since $x$ and $y$ were chosen arbitrarily. Moreover, for every $x,y \in X$, there exists a $w_\delta$-geodesic between $x$ and $y$.

Next, we prove $\#B_R(x) < \infty$ for all $x \in X$ and $R>0$.	
Suppose  $\# B_R(x) = \infty$ for some $R>0$, $x\in X$. 
Then, there are infinitely many $w_\delta$-geodesics $\gamma$ starting from $x$ with $w_\delta$-length $l_{w_\delta} (\gamma) \leq R$ since there exists a $w_\delta$-geodesic between $x$ and $y$ for all $y$. Thus, we can apply Lemma~\ref{Lemma path},and hence, there exists an infinite path $\gamma_\infty$ with $l_{w_\delta}(\gamma_\infty) \leq R$ such that $(\gamma_\infty(0),\ldots,\gamma_\infty(n))$ is $w_\delta$-geodesic for all $n$. Consequently, $\gamma_\infty$ is an infinite geodesic with finite length. This is a contradiction to geodesic completeness.\smallskip

(ii) $ \wedge $ (iii) $ \Rightarrow $ (iv): Essential local finiteness of $ w $ is immediate by (ii). Since every $ w $-geodesic is a $ \de_{w} $-geodesic, we infer that geodesic completeness assumed in (iii) implies  $ w $-geodesic completeness. \smallskip

(iv) $ \Rightarrow $ (ii): By Lemma~\ref{lf&geodcompl=>complete} for every $ x,y\in X $ there exists a $ w $-geodesic from $ x $ to $ y $. So, for any Cauchy-sequence $ (x_{k}) $  there are geodesics $ \gm_{k} $ from $ x_{0} $ to $ x_{k} $. If $ (x_{k}) $ does not admit a limit, then the set of these geodesics $\mathcal{P} :=\{\gm_{k}\} $ is infinite. Moreover,  $ (x_{k}) $ is bounded as a Cauchy-sequence, so, there exists an infinite $ w $-geodesic by Lemma~\ref{Lemma path} of finite length. This contradicts the geodesic completeness of $ w $.\smallskip

Lemma~\ref{Lemma discrete} implies discreteness of $\delta$ and Lemma\ref{Lemma geodesic} or Lemma~\ref{lf&geodcompl=>complete} show the existence of paths.

We still have to prove $w_\delta(x,y) = \sup\{\widetilde w(x,y)\mid\delta_{\widetilde w} = \delta\}$ for all $x,y \in X$. 
In the proof of (i) $\Rightarrow$ (ii) we have already shown that $\widetilde w$ is essentially locally finite if $\widetilde w$ generates $\delta$. Moreover, we showed
 $\widetilde w \leq w_\delta$ for all essentially locally finite $\widetilde w$ which generate $\delta$ in (ii) $\Rightarrow$ (iii). Furthermore, we have proven that $w_\delta$ generates $\delta$ in (iii) $\Rightarrow$ (i). Putting these three observations together, we obtain $w_\delta(x,y) = \sup\{\widetilde w(x,y)\mid \delta_{\widetilde w} = \delta\}$.  \smallskip
 
 This finishes the proof.	
\end{proof}
\section{Resistance metric}\label{resistance}
In this section we show that for a graph the resistance metric is a path metric induced by the graph structure if and only if  the graph is a block graph. Moreover, we can even characterize when the resistance metric is a path metric induced by the inverse edge weights.

Let $ b $ be a graph over a discrete set $ X $ which is a symmetric map $ b:X\times X\to[0,\infty) $ with zero diagonal that satisfies
\begin{align*}
\sum_{y\in X}b(x,y)<\infty,\qquad x\in X.
\end{align*}
Let $ C(X) $ be the real valued functions and let  $ Q:C(X)\to[0,\infty] $, 
\[
Q(f) := \frac 1 2 \sum_{x,y\in X} b(x,y)(f(x)-f(y))^2=\sum_{x\in X}\Gamma(f)(x),
\]
where
\begin{align*}
\Gamma(f)(x):= \frac 1 2 \sum_{y\in X} b(x,y)(f(x)-f(y))^2,\qquad x\in X.
\end{align*}

The following proposition is well known and can be shown by standard arguments found in \cite{LP16,JP,Kig01,GHKLW15}.

\begin{pro}[Resistance metric]\label{p:res}
The map $ R $ on $ X\times X $ given by
\[
R(x,y) := \sup \{(f(y)-f(x))^2  \mid   Q(f) = 1 \}= \sup_{f:X\to\R}  \frac{(f(y)-f(x))^{2}}{Q(f)}
\]
for $  x,y\in X $ is a metric. Moreover if $b$ is connected, the supremum in the definition is a unique maximum which is  assumed in a  function which is harmonic outside of $ x $ and $ y $, i.e., a function $ f $ satisfying $ Lf(v) := \sum_{w}b(v,w)(f(v)-f(w))=0 $ for all $ v\in X \setminus \{x,y \}$.	
\end{pro}
The metric $ R $ introduced in the proposition above is called the \emph{resistance metric}. Furthermore, let $ \de_{1/b} $ be the path metric generated by the weight $ 1/b $. If $b$ is a tree, see e.g. \cite[Lemma~8.1]{GHKLW15}, then
\begin{align*}
R=\de_{1/b}.
\end{align*}
So, it is a natural question whether there are other situations when $ R $ is a path metric induced by the graph structure. 
To answer this question, we first characterize sharpness of the triangle inequality of the resistance metric.

\begin{pro}\label{thm:resTriangle}
	Let $b$ be a graph over $ X $ and let $x,y,z \in X$. The following statements are equivalent:
	\begin{enumerate}
		\item[(i)] $ R(x,z) = R(x,y) + R(y,z). $
		\item[(ii)] All paths from $x$ to $z$ pass through $y$.
	\end{enumerate} 
\end{pro}

\begin{proof}
Without obstruction,	we assume that $b$ is connected.

	(i) $\Rightarrow$ (ii): We show the implication by contraposition. Let $f$ be a function satisfying $Q(f)=1$ and $(f(x)-f(z))^2=R(x,z)$ and $f(x)>f(z)$. Then, $f(x)>f(y)>f(z)$.
	We write
	$f_x:= f \vee f(y)$ and $f_z:=f \wedge f(y)$. By Young's inequality and since $1=Q(f) \geq Q(f_x) + Q(f_y)$, we have
	\[
	\frac{(f_x(x) - f_x(y))^2}{Q(f_x)} + 	
	\frac{(f_z(z) - f_z(y))^2}{Q(f_z)} \geq 
	\frac{(f(x) - f(z))^2}{Q(f_x) + Q(f_z)}
	\geq R(x,z).
	\]
	By definition, we have
	$R(x,y) \geq \frac{(f_x(x) - f_x(y))^2}{Q(f_x)}$ and
	$R(y,z) \geq \frac{(f_z(z) - f_z(y))^2}{Q(f_z)}$.
	Suppose there is a path from $x$ to $z$ not passing through $y$. Then, there are vertices $v,w\in X$ with $ b(v,w)>0 $ on this path such that $f(v) > f(y)\ge f(w) $ and $y \neq w$. 	
	Applying Proposition~\ref{p:res} gives  $Lf(w) = 0$ and thus,
	$Lf_z(w) = Lf(w) + b(v,w)(f(v) - f(y)) \neq 0$, implying
     $R(y,z) > \frac{(f_z(z) - f_z(y))^2}{Q(f_z)}$. Consequently,
	\[
	R(x,y)+R(y,z)  
	>     \frac{(f_x(x) - f_x(y))^2}{Q(f_x)} + 	
	\frac{(f_z(z) - f_z(y))^2}{Q(f_z)}  
	\geq  R(x,z).
	\]

	(ii) $\Rightarrow$ (i): Let $ x,y,z\in X $. Assume every path from $ x $ to $ z $ passes through $ y $.
	We define
	\[
	S_x:=\{w\in X \setminus\{y\}\mid \mbox{ there exists a path from } x \mbox{ to } w \mbox{ not passing through } y\}.
	\]
	Since every path from $x$ to $z$ passes through $y$, we have
	\[
	S_x=\{w\in X \setminus\{y\}: \mbox{ every path from } z \mbox{ to } w \mbox{  passes through } y\}.
	\]
	Analogously, we define
	\begin{align*}
	S_z &:=\{w\in X \setminus\{y\}: \mbox{ there exists a path from } z \mbox{ to } w \mbox{ not passing through } y\} \\
	&= \{w\in X \setminus\{y\}: \mbox{ every path from } x \mbox{ to } w \mbox{  passes through } y\}.
	\end{align*}
	Hence, $$ S_x\cap S_z = \emptyset \qquad \mbox{and}\qquad
	b(S_x\times S_z) =0,
	 $$ i.e., $ b(v,w)=0 $ for all $ v\in S_{x},w\in S_{z} $.
	Obviously, $x \in S_x$ and $z \in S_z$.
	
	We take a function $f_x$ such that. $Q(f_x)=1$ and $f_x(y)=0$ and $f_x(x)^2=R(x,y)$ and $f_x(x)>0$. Analogously, we take a function
	$f_z$, such that  $Q(f_z)=1$ and $f_z(y)=0$ and $f(z)^2=R(z,y)$ and $f_z(z)>0$.
	
	We claim that $\supp (f_x) \subset S_x$.
	To prove the claim, we set $\widetilde{f_x}:=f_x \cdot 1_{S_x}$ and observe that it suffices to show $Q(\widetilde{f_x})\leq Q(f_x)$ due to uniqueness of the minimizer following from Proposition~\ref{p:res}. To this end, observe that every neighbor of $ w \in S_{x}$ is either in $ S_{x} $ or $ y $. Since $ f_{x}(y)=0 $ we infer
		\begin{align*}
		\Gamma (\widetilde{f_x})(w) &= \Gamma( f_x)(w),\qquad w \in S_x.
		\end{align*}
		Moreover, since  $S_x\cap S_z = \emptyset$ and  $b(S_x\times S_z) = 0$ we have $ \widetilde{f}_{x}=0 $ on $ S_{z} $ and all neighbors of vertices in $ S_{z} $. Thus,
		\begin{align*}
		\Gamma (\widetilde{f_x})(w) &= 0,\qquad w \in S_z 
		\end{align*}
		Finally, on all neighbors of $ y $ in $  S_{x}$ the functions $ f_{x} $ and $ \widetilde f_{x} $ agree, for all other neighboring vertices $ w $ we have $ (\widetilde f_{x}(y)-\widetilde f_{x}(w))^{2}=0 $ and, hence,
	\begin{align*}
	\Gamma (\widetilde{f_x})(y) &\leq \Gamma (f_x)(y).  
	\end{align*}
	Thus, $Q(\widetilde{f_x})\leq Q(f_x)$ and, therefore, $\supp (f_x) \subset S_x$.  
	Analogously, one shows   $\supp (f_z) \subset S_z$.	As a consequence the function
	\begin{align*}
	 f:=  \frac{1}{\sqrt{f_{x}(x)^{2}+f_{z}(z)^{2}}}(f_{x}(x) f_{x}- f_{z}(z)f_{z})
	\end{align*} 
	satisfies 
	\begin{align*}
	Q(f)= \frac{1}{{f_{x}(x)^{2}+f_{z}(z)^{2}}}(f_{x}(x)^{2} Q(f_{x})+ f_{z}(z)^{2}Q(f_{z})=1
	\end{align*}
	and, therefore, since we have
	 $ f(x)=f_{x}(x)^{2}/ \sqrt{f_{x}(x)^{2}+f_{z}(z)^{2}}$ and  $ f(z)=-f_{z}(z)^{2}/ \sqrt{f_{x}(x)^{2}+f_{z}(z)^{2}}$, we obtain
	\begin{align*}
	R(x,y)\ge (f(x)-f(z))^{2}
	=f_{x}(x)^{2}+f_{z}(z)^{2}=R(x,y)+R(y,z).
	\end{align*}
	Since $ R $ satisfies the triangle inequality statement (i) follows.
\end{proof}

\begin{thm}
	Let $ b $ be a connected graph over $ X $. The following statements are equivalent:
	\begin{itemize}
		\item [(i)] $ R=\de_{1/b} $
		\item [(ii)] $ b $ is a tree.
	\end{itemize}
\end{thm}
\begin{proof}
(ii) $ \Longrightarrow $ (i): This follows from the theorem above or \cite[Lemma~8.1]{GHKLW15} (see the arxiv version for a detailed proof).

(i) $ \Longrightarrow $ (ii): Assume there exists a cycle. Let $ x $ be a vertex on this cycle. Denote by $ C $ the set of all neighbors $ y $ of $ x $ such that the edge $ (x,y) $ is contained in a cycle. Since $ b $ is summable about $ x $ there exists $ y\in C $ such that $ b(x,y) $ is the maximum of $ b(x,\cdot) $ on $ C $. Then,  $ d_{1/b}(x,y)=1/b(x,y) $. Moreover, any function $ f $ on $ X $ such that $ f(x)=1 $ and $ f(y)=0 $ satisfies $ Q(f)\ge b(x,y) $ and, therefore, $ R(x,y)\ge \de_{1/b}(x,y) $. 
In order to achieve equality, we need that such a function $ f $ satisfies  $ f(v)=f(w) $ for all $ v,w\in X $  with $ b(v,w)>0 $ and $ \{x,y\} \neq\{v,w\}$. This however, is impossible as the edge $ (x,y) $ is contained in a cycle.
\end{proof}

In the above corollary, we can also weaken the assumption $R=\delta_{1/b}$ by only requiring that the weight is supported only on the edges.
We say a weight $w$ is \emph{compatible} to a graph $b$ if $w(x,y) = \infty$ whenever $b(x,y)=0$ and $x \neq y$.
It turns out that the existence of a compatible weight inducing the resistance metric is related to\emph{ block graphs}, i.e., graphs where any two vertices are connected by a unique induced path.

\begin{thm}
	Let $b$ be a locally finite graph over $X$.
Then, the following statements are equivalent:
\begin{itemize}
	\item [(i)] $ R=\de_{w}$ for some $w$ compatible to $b$.
	\item [(ii)] The graph $b$ is a block graph.
\end{itemize}	
\end{thm}

\begin{proof}
(ii) $ \Longrightarrow $ (i):
We choose $$w(x,y)=\begin{cases}
R(x,y)&: x\sim y\mbox{ or } x=y\\
\infty &:\mathrm{otherwise}.
\end{cases}$$
By triangle inequality, $\delta_w \geq R$. Moreover, $\delta_w(x,y) = R(x,y)$ if $x\sim y$.
It remains to show $\delta_w(x,y) \leq R(x,y)$ for non-adjacent $x,y$.
Let $(x_0,\ldots,x_n)$ be the unique induced path from $x$ to $y$.
Then, every path from $x$ to $y$ must pass through all $x_i$ for $i=0,\ldots, n$.
Thus, we can apply Theorem~\ref{thm:resTriangle} to conclude
\[
R(x,y)=\sum_{i=1}^n R(x_i,x_{x-1})= \sum_{i=1}^n w(x_i,x_{x-1}) \geq \delta_w(x,y)
\]
which proves $R=w_\de$.

(ii) $ \Longrightarrow $ (i):
We assume that $R=\delta_w$ for some $w$ compatible to the graph.
Let $x \neq y \in X$.
We aim to show that there is a unique induced path from $x$ to $y$.
We distinguish two cases.

Case one: There is a $w$-geodesic $(x_0,\ldots,x_n)$ from $x$ to $y$.
Suppose there exists an induced path $\gamma \neq (x_0,\ldots,x_n)$ from $x$ to $y$.
Then, there exists $i \in \{0,\ldots n\}$ such that  $\gamma$ does not pass through $x_i$.
Hence by Theorem~\ref{thm:resTriangle},
$R(x,y) < R(x,x_i) + R(x_i,y)$. But this is a contradiction to the fact that  $(x_0,\ldots,x_n)$ is a geodesic. So in the first case, there exists a unique induced path from $x$ to $y$.

Case two: There is no $w$-geodesic from $x$ to $y$.
Then, there exists a sequence of paths $\gamma_n$ from $x$ to $y$ with $l_w(\gamma_n) \to R(x,y)$ as $n \to \infty$.
By local finiteness and by Lemma~\ref{Lemma path}, there exists an infinite path $\gamma_\infty$ such that  for all $k \in \N$ there exists $n \in \N$ such that  $\gamma_n(i)=\gamma_\infty(i)$ for $i=0,\ldots, k$.
Then, $\gamma_\infty$ is a $w$-geodesic by triangle inequality.
Let $k$ be large such that  $\gamma_1$ does not pass through $\gamma_\infty(k-1)$.
Let $\gamma_n$ be a path such that  $\gamma_\infty(i)=\gamma_n(i)$ for all $i \leq k$.
Let $$W:= \ran(\gamma_1) \cup (\ran (\gamma_n) \setminus  \{\gamma_\infty(0),\ldots, \gamma_\infty(k-1)\}).$$
Then as a union of paths sharing the vertex $y$, the by $W$ induced subgraph is connected. Hence, there exists a path from $x$ to $\gamma_\infty(k)$ within $W$. Observe, $\gamma_\infty(k-1) \notin W$.
Thus, by Theorem~\ref{thm:resTriangle},
$R(x,\gamma_\infty(k))< R(x,\gamma_\infty(k-1)) + R(\gamma_\infty(k-1),\gamma_\infty(k))$.
But this is a contradiction to the fact that $\gamma_\infty$ is a geodesic.
Hence, the second case can not occur.
The complete case distinction finishes the proof.
\end{proof}

\bibliographystyle{alpha}
	\bibliography{literature}

\newcommand{\etalchar}[1]{$^{#1}$}
\def\cprime{$'$} \def\cprime{$'$}
\begin{thebibliography}{CdVTHT11}

\bibitem[BBI01]{BBI01}
Dmitri Burago, Yuri Burago, and Sergei Ivanov.
\newblock {\em A course in metric geometry}, volume~33 of {\em Graduate Studies
  in Mathematics}.
\newblock American Mathematical Society, Providence, RI, 2001.

\bibitem[BHK13]{BHK13}
Frank Bauer, Bobo Hua, and Matthias Keller.
\newblock On the {$l^p$} spectrum of {L}aplacians on graphs.
\newblock {\em Adv. Math.}, 248:717--735, 2013.

\bibitem[BHY17]{BHY17}
Frank Bauer, Bobo Hua, and Shing-Tung Yau.
\newblock Sharp {D}avies-{G}affney-{G}rigor'yan lemma on graphs.
\newblock {\em Math. Ann.}, 368(3-4):1429--1437, 2017.

\bibitem[CdVTHT11]{CdVTT11}
Yves Colin~de Verdi{\`e}re, Nabila Torki-Hamza, and Fran{\c{c}}oise Truc.
\newblock Essential self-adjointness for combinatorial {S}chr\"odinger
  operators {II}---metrically non complete graphs.
\newblock {\em Math. Phys. Anal. Geom.}, 14(1):21--38, 2011.

\bibitem[Fol11]{Fol11}
Matthew Folz.
\newblock Gaussian upper bounds for heat kernels of continuous time simple
  random walks.
\newblock {\em Electron. J. Probab.}, 16:no. 62, 1693--1722, 2011.

\bibitem[Fol14]{Fol14}
Matthew Folz.
\newblock Volume growth and stochastic completeness of graphs.
\newblock {\em Trans. Amer. Math. Soc.}, 366(4):2089--2119, 2014.

\bibitem[GHK{\etalchar{+}}15]{GHKLW15}
Agelos Georgakopoulos, Sebastian Haeseler, Matthias Keller, Daniel Lenz, and
  Rados{\l}aw~K. Wojciechowski.
\newblock Graphs of finite measure.
\newblock {\em J. Math. Pures Appl. (9)}, 103(5):1093--1131, 2015.

\bibitem[HKMW13]{HKMW13}
Xueping Huang, Matthias Keller, Jun Masamune, and Rados{\l}aw~K. Wojciechowski.
\newblock A note on self-adjoint extensions of the {L}aplacian on weighted
  graphs.
\newblock {\em J. Funct. Anal.}, 265(8):1556--1578, 2013.

\bibitem[HKW13]{HKW13}
Sebastian Haeseler, Matthias Keller, and Rados{\l}aw~K. Wojciechowski.
\newblock Volume growth and bounds for the essential spectrum for {D}irichlet
  forms.
\newblock {\em J. Lond. Math. Soc. (2)}, 88(3):883--898, 2013.

\bibitem[HR31]{HopfRinow}
H.~Hopf and W.~Rinow.
\newblock Ueber den {B}egriff der vollst\"andigen differentialgeometrischen
  {F}l\"ache.
\newblock {\em Comment. Math. Helv.}, 3(1):209--225, 1931.

\bibitem[Hua14]{Hua14}
Xueping Huang.
\newblock Escape rate of {M}arkov chains on infinite graphs.
\newblock {\em J. Theoret. Probab.}, 27(2):634--682, 2014.

\bibitem[Jos17]{Jost17}
J\"urgen Jost.
\newblock {\em Riemannian geometry and geometric analysis}.
\newblock Universitext. Springer, Cham, seventh edition, 2017.

\bibitem[JP]{JP}
P.~Jorgensen and E.~Pearse.
\newblock {\em Operator theory of electrical resistance networks}.
\newblock Springer.

\bibitem[Kel15]{Ke15}
Matthias Keller.
\newblock Intrinsic metrics on graphs: a survey.
\newblock In {\em Mathematical technology of networks}, volume 128 of {\em
  Springer Proc. Math. Stat.}, pages 81--119. Springer, Cham, 2015.

\bibitem[Kig01]{Kig01}
Jun Kigami.
\newblock {\em Analysis on fractals}, volume 143 of {\em Cambridge Tracts in
  Mathematics}.
\newblock Cambridge University Press, Cambridge, 2001.

\bibitem[LP16]{LP16}
Russell Lyons and Yuval Peres.
\newblock {\em Probability on trees and networks}, volume~42 of {\em Cambridge
  Series in Statistical and Probabilistic Mathematics}.
\newblock Cambridge University Press, New York, 2016.

\bibitem[Mil11]{Mil11}
Ognjen Milatovic.
\newblock Essential self-adjointness of magnetic {S}chr\"odinger operators on
  locally finite graphs.
\newblock {\em Integral Equations Operator Theory}, 71(1):13--27, 2011.

\end{thebibliography}

\end{document}